\documentclass[preprint,review,10pt]{amsart}
\usepackage{amsmath,amssymb,amsfonts,xspace,mathrsfs}
\newtheorem{theorem}{Theorem}[section]
\newtheorem{lemma}[theorem]{Lemma}

\theoremstyle{definition}
\newtheorem{definition}[theorem]{Definition}
\newtheorem{example}[theorem]{Example}

\newtheorem{proposition}[theorem]{Proposition}

\theoremstyle{remark}
\newtheorem{remark}[theorem]{Remark}

\numberwithin{equation}{section}

\begin{document}

\title{$\mathcal{Q}$-filters of bounded lattices}

\author{Mahdi Anbarloei}
\address{Department of Mathematics, Faculty of Sciences,
Imam Khomeini International University, Qazvin, Iran.
}

\email{m.anbarloei@sci.ikiu.ac.ir}





\begin{abstract}
 In this paper, we introduce and study the notion of  $\mathcal{Q}$-filters in  bounded lattices. 

\end{abstract}
\maketitle

\section{Introduction} 
Throughout this paper, we assume that all lattices are bounded, i.e., they have a least element $0$ and a greatest element $1$.  By treating lattices as generalizations of rings, it is natural to examine which ring properties hold in lattice theory. Since the lack of subtraction prevents many ring results from having direct analogues, overcoming this conceptual limitation remains a highly pursued goal in the literature.

As cornerstone concepts in abstract algebra, prime ideals continue to motivate numerous studies focused on expanding and generalizing their properties. In \cite{Mimouni}, Mimouni introduced the notion of $Q$-ideals in a commutative ring unifying  and some ideals, such as prime ideals, primary ideals, $n$-ideals and $J$-ideals in a single framework. Let $A \subseteq Q$ be ideals of a commutative ring $R$. Then, $A$ is called a $Q$-ideal if for every $a,b \in R$, $ab \in A$ and $a \notin Q$ implies that $b \in A$. Later, Smach extended this concept to  $S$-$Q$-ideals in commutative rings via a multiplicative subset in \cite{Smach}. Furthermore, he defined the notion of almost $S$-$Q$-ideals in the same paper. 

 To date, different types of filters have been created to allow a complete understanding of general lattice structures. Recently, the notion of $J$-filter in a bounded lattice was introduced by Atani in \cite{Atani6}. Our aim in this paper is to introduce and
study the notion of $Q$-filters of bounded lattices, 
thereby establishing   a unified framework for prime filters and $J$-filters. Next, we propose a generalization of  this notion with the help of a $\vee$-closed subset $S$, called  $S$-$Q$-filters. Moreover, we present a weak version of $S$-$Q$-filters termed almost $S$-$Q$-filters.

 Among the diverse results established in this work, we show that the filter $\{1\}$ is a $Q$-filter of a lattice $\mathcal{L}$ if and only if $Q \neq \{1\}$ contains the set of all identity joins of $\mathcal{L}$ in Proposition \ref{1}.   Theorem \ref{4} presents conditions under which a lattice is local. In theorem \ref{5}, we conclude that if $\mathcal{P} \vee \mathcal{Q}$ is a $\mathcal{Q}$-filter of $\mathcal{L}$, then  $ \mathcal{Q}$ contains $\mathcal{P}$ or $\mathcal{Q}$ is a prime filter of $\mathcal{L}$ with $\mathcal{Q}=\mathcal{Q} \vee \mathcal{P}.$ While every $\mathcal{Q}$-filter of $\mathcal{L}$ is an $S$-$\mathcal{Q}$-filter, Example \ref{exasq} shows that the converse is not true in general. Moreover, we examine the behavior of
these notions  under some lattice-theoretic constructions.

\section{Preliminaries}
In this section, we briefly recall some preliminary definitions and basic results  from \cite{Birkhoff} and other research works that will be utilized throughout this paper.

A poset $(\mathcal{L}, \le)$ is a \textit{lattice} if   each pair $u, v \in \mathcal{L}$ has a  greatest lower bound (briefly, g.l.b.) and a   least upper bound (briefly, l.u.b.), denoted by $u \wedge v$ and $u \vee v$, respectively. 
A lattice $\mathcal{L}$ is called \textit{complete} if every subset of $\mathcal{L}$ has a greatest lower bound and least upper bound in $\mathcal{L}$. It follows immediately that every nonvoid complete lattice contains a least element $0$ and a greatest element $1$. In this case, $\mathcal{L}$ is a bounded lattice.
Let $\mathcal{L}$ be a lattice. A non-empty subset  $F$ of $\mathcal{L}$ is said to be a \textit{filter}, if $u \wedge v \in F$ for all $u, v \in F$, and for $u \in F$ and $w \in \mathcal{L}$, $u \le w$ implies $w \in F$. Assume that $\mathcal{L}$ is a lattice with $1$. Then,  $1$ is an element of   every filter of $\mathcal{L}$ and $\{1\}$ is a filter of $\mathcal{L}$. By $\mathcal{F}(\mathcal{L})$, we mean the set of all filters of $\mathcal{L}$. Let $u \in \mathcal{L}$.  A complement of $u$ in $\mathcal{L}$ is an element $v \in \mathcal{L}$ such that $u \vee v = 1$ and $u \wedge v= 0$.   $\mathcal{L}$ is called a complemented lattice if any element of $\mathcal{L}$ has a complement in $\mathcal{L}$.

\begin{definition}
\cite{Birkhoff} Let  $F$ be a proper filter of $\mathcal{L}$. We say that 
\begin{itemize}
 \item[\rm(1)] $F$ is \textit{prime} if $u \vee v \in F$ implies that $u \in F$ or $v \in F$.
\item[\rm(2)]  $F$ is  \textit{maximal} if whenever $H$ is a filter in $\mathcal{L}$ such that $F \subsetneqq H$, then $H = \mathcal{L}$.
\end{itemize}
\end{definition}

\begin{definition} \cite{Atani1}
Let $\mathcal{L}$ be a lattice and $X \subseteq \mathcal{L}$.  The filter generated by $X$, denoted by $\text{T}(A)$,  is defined as 
$$\text{T}(X) = \bigcap \{ F \in  \mathcal{F}(\mathcal{L}) \mid   X \subseteq F \}.$$
\end{definition}

A filter $F$ of a lattice $\mathcal{L}$ is   finitely generated if $F = T(X)$ for some finite subset $X$ of $\mathcal{L}$.
\begin{lemma} \cite{Atani1}
Let $X $ be a non-empty of  $\mathcal{L}$. Then, \[T(X) = \{u \in \mathcal{L} \mid  u_1 \wedge u_2 \wedge \dots \wedge u_n \le u \text{ for some } u_i \in X \; (1 \le i \le n)\}.\]
\end{lemma}

\begin{definition}
\cite{Atani2} Let  $S$ of $\mathcal{L}$ is \textit{$\vee$-closed} if  $s_1 \vee s_2 \in S$ for all $s_1, s_2 \in S$ and $0 \in S$.
\end{definition}
Let $F$ be a prime filter of $\mathcal{L}$. Then, $\mathcal{L} \setminus F$ is a $\vee$-closed subset of $\mathcal{L}$. 
\section{$\mathcal{Q}$-filters}
In this section, we introduce and investigate the  concept of $\mathcal{Q}$-filters in a bounded lattice. We initiate our study with the following definition.
\begin{definition}
Let $\mathcal{P}$ and $ \mathcal{Q}$ be  proper filters of $\mathcal{L}$. We say that $\mathcal{P}$ is a {\it $\mathcal{Q}$-filter} of $\mathcal{L}$  if  for all $u,v \in \mathcal{L}$,   $u \vee v \in \mathcal{P}$ and $u \notin \mathcal{Q}$ imply that   $v \in \mathcal{P}$.
\end{definition}
\begin{example} \label{exa1}
\begin{enumerate}
\item Let $\mathcal{L}=\{0,u,v,w,1\}$ with the ralations $u < v, u \wedge w=0$ and $v \vee w=1$. Set $\mathcal={P}=\{1,a\}$ and $\mathcal{Q}=\{1,u,v\}$. Then, $\mathcal{P}$ is a $\mathcal{Q}$-filter of $\mathcal{L}$.
\item Consider the lattice $\mathcal{L} = \{0, u, v, w, 1\}$   in which  $0 \le u \le w \le 1$, $0 \le v \le w \le 1$, $u \vee v = w$, and $u \wedge v = 0$. In the lattice $\mathcal{L}$, $\mathcal{P}=\{1.w\}$, $\mathcal{Q}_1=\{1,w,u\}$ and $\mathcal{Q}_2=\{1,w,v\}$ are filters.   Although  $u \vee v =w\in \mathcal{P}$ and $v \notin \mathcal{Q}_1$, $u \notin  \mathcal{P}$. This implies that $\mathcal{P}$ is not a $\mathcal{Q}_1$-filter of $\mathcal{L}$. Similarly, one  can  check the filter $\mathcal{P}$  is  not a  $\mathcal{Q}_2$-filter  of $\mathcal{L}$.
 \end{enumerate}
\end{example}
\begin{proposition}
If $\mathcal{P}$ is a {\it $\mathcal{Q}$-filter} of $\mathcal{L}$, then $\mathcal{P} \subseteq  \mathcal{Q}$. 
\end{proposition}
\begin{proof}
On the contrary, assume that $\mathcal{P} \nsubseteq  \mathcal{Q}$. Then, there exists $u \in \mathcal{P}$ such that  $u \notin  \mathcal{Q}$. Since $\mathcal{P}$ is a {\it $\mathcal{Q}$-filter} of $\mathcal{L}$, $u =u \vee 0 \in \mathcal{P}$ and $u \notin  \mathcal{Q}$, we get $0 \in \mathcal{P}$, a contradiction. Therefore, we conclude that $\mathcal{P} \subseteq  \mathcal{Q}$. 
\end{proof}
Recall from \cite{Atani2} that an element $u \in \mathcal{L}$ is  the {\it identity join} of $\mathcal{L}$ if there exists $1 \neq v \in \mathcal{L}$ such that $u \vee v =1$. By $\text{Id}(\mathcal{L})$, we mean the set of all identity joins of $\mathcal{L}$. Moreover, 
a lattice $\mathcal{L}$ containing  $1$ is  an {\it $\mathcal{L}$-domain} if for any $u, v \in \mathcal{L}$, $ u \vee v = 1$ implies $u = 1$ or $ v = 1$.
 Consequently, $\mathcal{L}$ is an $\mathcal{L}$-domain if and only if $\{1\}$ is a prime filter of $\mathcal{L}$.
 \begin{proposition} \label{1}
 Let $\mathcal{Q}$ be a filter of $\mathcal{L}$ such that $\mathcal{Q} \neq \{1\}$. The filter $\{1\}$ is a $\mathcal{Q}$-filter of $\mathcal{L}$ if and only if $\text{Id}(\mathcal{L}) \subseteq \mathcal{Q}$. Furthermore,  if $\mathcal{L}$ is an $\mathcal{L}$-domain, then  $\{1\}$ is a $\mathcal{Q}$-filter of $\mathcal{L}$ for every filter $\mathcal{Q}$ of $\mathcal{L}$.
 \end{proposition}
 \begin{proof}
$\Longrightarrow$ Let   $v \in \text{Id}(\mathcal{L})$ but $v \notin  \mathcal{Q}$. Then, there exists $1 \neq u \in \mathcal{L}$ such that $u \vee v =1$. Since  $\{1\}$ is a $\mathcal{Q}$-filter of $\mathcal{L}$ and $v \notin \mathcal{Q}$, we get $u \in \{1\}$, a contradiction. Hence, $\text{Id}(\mathcal{L}) \subseteq \mathcal{Q}$.

$\Longleftarrow$ Let $u \vee v \in \{1\}$ for $u,v \in \mathcal{L}$ and $v \notin \mathcal{Q}$. This implies that $v \notin \text{Id}(\mathcal{L})$. Therefore, we conclude that $u=1$ as $u \vee v =1$. Thus, the filter $\{1\}$ is a $\mathcal{Q}$-filter of $\mathcal{L}$.

Now, we show that $``\text{furthermore}"$ statement holds. Let $\mathcal{Q}$ be an arbitrary filter of $\mathcal{L}$. Assume that $u \vee v \in \{1\}$ such that $ v \notin \mathcal{Q}$. Since $\mathcal{L}$ is an $\mathcal{L}$-domain  and $v \neq 1$, we get $u=1$. Consequently, $\{1\}$ is a $\mathcal{Q}$-filter of $\mathcal{L}$.
 \end{proof}
 
\begin{theorem} \label{2}
Let $\mathcal{P}$ and $ \mathcal{Q}$ be   filters  of $\mathcal{L}$ such that  $\mathcal{P} \neq \{1\}$.  The following statements are equivalent: 
\begin{itemize} 
\item[\rm(i)]~ $\mathcal{P}$ is a  $\mathcal{Q}$-filter of $\mathcal{L}$;
\item[\rm(ii)]~  For any two filters $\mathcal{I}$ and $\mathcal{J}$  of $\mathcal{L}$, $\mathcal{I} \vee \mathcal{J} \subseteq \mathcal{P}$ and $\mathcal{I} \nsubseteq \mathcal{Q}$ imply that   $\mathcal{J}  \subseteq \mathcal{P} $.
\end{itemize}
\end{theorem}
\begin{proof}
(i) $\Longrightarrow$ (ii) Let $\mathcal{P}$ be a  $\mathcal{Q}$-filter of $\mathcal{L}$. Assumee that $\mathcal{I} \vee \mathcal{J} \subseteq \mathcal{P}$ for filters  $\mathcal{I}$ and $\mathcal{J}$  of $\mathcal{L}$ such that $\mathcal{I}  \nsubseteq  \mathcal{Q}$. Then, there exists $u \in \mathcal{I}$ such that $ u \notin  \mathcal{Q}$. Now, take any $v \in \mathcal{J}$.  Since  $\mathcal{P}$ is a  $\mathcal{Q}$-filter of $\mathcal{L}$, $ u \vee v \in \mathcal{I} \vee \mathcal{J} \subseteq \mathcal{P}$ and $u \notin \mathcal{Q}$, we get  $v \in \mathcal{P}$ and so $\mathcal{J}  \subseteq \mathcal{P} $.   Hence, (ii) holds.

(ii) $\Longrightarrow$ (i) Let $u \vee v \in \mathcal{P}$ for $u,v \in \mathcal{L}$ and $u \notin \mathcal{Q}$. Put $\mathcal{I}=T(\{u\})$ and $\mathcal{J}=T(\{v\})$.   Since  $\mathcal{I} \vee \mathcal{J} \subseteq \mathcal{P}$ and $\mathcal{I} \nsubseteq \mathcal{Q}$, we obtain $\mathcal{J} \subseteq \mathcal{P}$ or  by the assumption. This implies that $v \in \mathcal{P}$. Thus, $\mathcal{P}$ is a $\mathcal{Q}$-filter of $\mathcal{L}$.
\end{proof}
Let $J(\mathcal{L})$ is the intersection of all maximal filters of $\mathcal{L}$. A proper ﬁlter $\mathcal{P}$ of $\mathcal{L}$ is called a {\it $J$-ﬁlter} if whenever $u,v \in \mathcal{L}$ with $u \vee v \in \mathcal{P}$ and $v \notin J(\mathcal{L})$, then $u \in \mathcal{P}$ \cite{Atani6}.
\begin{theorem} \label{3}
Let $\mathcal{P} \neq \{1\}$ be   a filter of $\mathcal{L}$.  Then,  $\mathcal{P}$ is a  $J$-filter of $\mathcal{L}$ if and only if  $\mathcal{P}$ is an  $\mathcal{M}$-filter of $\mathcal{L}$ for every maximal filter $\mathcal{M}$ of $\mathcal{L}$.
\end{theorem}
\begin{proof}
$\Longrightarrow$ Let $\mathcal{P}$ be a  $J$-filter of $\mathcal{L}$ and $\mathcal{M}$ be an arbitrary filter of $\mathcal{L}$. Then,   we have $J(\mathcal{L}) \subseteq \mathcal{M}$. Assume that $u \vee v \in \mathcal{P}$ such that $ u \notin \mathcal{M}$. Therefore, we get  $v \in  \mathcal{P}$ as $\mathcal{P}$ is a  $J$-filter of $\mathcal{L}$ and $u \notin  J(\mathcal{L})$. Hence, $\mathcal{P}$ is an  $\mathcal{M}$-filter of $\mathcal{L}$.

$\Longleftarrow$ Assume that $\mathcal{P}$ is an  $\mathcal{M}$-filter of $\mathcal{L}$ for every maximal filter $\mathcal{M}$ of $\mathcal{L}$. Let $u \vee v \in \mathcal{P}$ and $u \notin J(\mathcal{L})$. Then, there is at least a maximal  filter, say $\mathcal{M}_0$,  that do not contain $u$.  By the hypothesis, $\mathcal{P}$ is an  $\mathcal{M}_0$-filter of $\mathcal{L}$. This implies that   $v \in \mathcal{P}$ as $u \vee v \in \mathcal{P}$ and $u \notin \mathcal{M}_0$. Thus, $\mathcal{P}$ is a  $J$-filter of $\mathcal{L}$.
\end{proof}

Recall from \cite{Atani6} that a
$\mathcal{L}$ is called a local lattice if it has exactly one maximal ﬁlter  containing all proper ﬁlters. 
 \begin{theorem} \label{4}
 Let  $\mathcal{L}$ is a  lattice. Then, $\mathcal{L}$ is a local lattice if and only if for every $ 0 \neq  a \in \mathcal{L}$, $T(\{a\})$ is an $\mathcal{M}$-filter for each maximal filter $\mathcal{M}$ containing $T(\{a\})$.
 \end{theorem}
 \begin{proof}
 $\Longrightarrow$  Suppose that $\mathcal{L}$ with unique maximal filter $\mathcal{M}$. Take any $ 0 \neq  a \in \mathcal{L}$. Let $ u \vee v \in T(\{a\})$ and $u \notin  \mathcal{M}$. Since $ \mathcal{M} \subsetneqq \mathcal{M} \wedge T(\{a\}$, we get $\mathcal{M} \wedge T(\{a\}= \mathcal{L}$ by maximality of $\mathcal{M}$. By Lemma 3.1 in \cite{Atani7}, we conclude that $T(\{a\}= \mathcal{L}$. Then, there exists $b \in  \mathcal{L}$, $0 =b \vee a$. This implies that $v=v \vee 0=v \vee b \vee a \in T(\{a\}$. Thus, $T(\{a\})$ is an $\mathcal{M}$-filter of $\mathcal{L}$.
 
$\Longleftarrow$  Let for every $ a \in \mathcal{L}$, $T(\{a\})$ be  an $\mathcal{M}$-filter for each maximal filter $\mathcal{M}$ containing $T(\{a\})$. On the contrary, assume that $\mathcal{L}$ is not a local lattice. Then, $\mathcal{L}$ has at least two maximal filters, say $\mathcal{M}_1$ and $\mathcal{M}_2$. Let  $u \in \mathcal{M}_1 \setminus  \mathcal{M}_2$ and $v \in \mathcal{M}_2 \setminus  \mathcal{M}_1$. By the hypothesis,   $T(\{u \vee v\})$ is an $\mathcal{M}_2$-filter of $\mathcal{L}$. Since $u \vee v \in T(\{u \vee v\})$ and $u \notin \mathcal{M}_2$, we get $v \in T(\{u \vee v\}) \subseteq \mathcal{M}_1$, a contradiction. Hence, $\mathcal{L}$ is local.
 \end{proof}
 \begin{theorem} \label{5}
Let $\mathcal{P}$ and $\mathcal{Q}$ be filters of $\mathcal{L}$. If $\mathcal{P} \vee \mathcal{Q}$ is a $\mathcal{Q}$-filter of $\mathcal{L}$, then $\mathcal{P} \subseteq \mathcal{Q}$ or $\mathcal{Q}$ is a prime filter of $\mathcal{L}$ such that  $\mathcal{Q}=\mathcal{Q} \vee \mathcal{P}.$
\end{theorem}
\begin{proof}
Let $\mathcal{P} \vee \mathcal{Q}$ be a $\mathcal{Q}$-filter of $\mathcal{L}$ such that $\mathcal{P} \nsubseteq \mathcal{Q}$.  Then, there exists  $u \in \mathcal{P}$ such that $u \notin \mathcal{Q}$. Let $v \in \mathcal{Q}$. So, $u \vee v  \in \mathcal{P} \vee \mathcal{Q}$. Since   $\mathcal{P} \vee \mathcal{Q}$ is a $\mathcal{Q}$-filter of $\mathcal{L}$ and $u \notin \mathcal{Q}$, we get $v \in \mathcal{P} \vee \mathcal{Q}$ which means $\mathcal{Q} \subseteq  \mathcal{P} \vee \mathcal{Q} \subseteq  \mathcal{Q}$. Hence, we conclude that $\mathcal{Q}=\mathcal{P} \vee \mathcal{Q}$ and so $\mathcal{Q} \varsubsetneq  \mathcal{P}$. Now, let $a \vee b \in Q$ for $a,b \in \mathcal{L}$ such that $a \notin \mathcal{Q}$. Since $\mathcal{Q}=\mathcal{P} \vee \mathcal{Q}$ is a $\mathcal{Q}$-filter of $\mathcal{L}$, we get   $b  \in  \mathcal{Q}$.  This means that $\mathcal{Q}$ is a prime filter of $\mathcal{L}$.
\end{proof}
Let  $\mathcal{I}$ be  a filter of a lattice ($\mathcal{L}$, $\le$). Let  us define   the  relation $\sim$ on $\mathcal{L}$ as  $u \sim v$ if and only if there exist $a, b\in \mathcal{I}$ satisfying $u \wedge a = v \wedge b$. In this case,  $\sim$ is an equivalence relation on $\mathcal{L}$.  Let  $u \wedge \mathcal{I}$ be the equivalence class of $u$   and let $ \mathcal{L}/ \mathcal{I}$ be the collection of all equivalence classes. Now, consider the partial order $\le_Q$ on $ \mathcal{L}/ \mathcal{I}$ as follows: for each $u \wedge \mathcal{P}, v \wedge \mathcal{I} \in   \mathcal{L}/ \mathcal{I}$,  $u \wedge \mathcal{I} \le_Q v \wedge \mathcal{I}$ if and only if $u \le v$. Then $(\mathcal{L}/ \mathcal{I},\le_Q)$ is a lattice with $(u \wedge \mathcal{I} ) \vee_Q (v \wedge \mathcal{I} ) = (u \vee v) \wedge \mathcal{I}$ and $(u \wedge \mathcal{I} ) \wedge_Q (v \wedge \mathcal{I} ) = (u \wedge v) \wedge \mathcal{I}$ for all elements $u \wedge \mathcal{I}, v \wedge \mathcal{I} \in \mathcal{L}/ \mathcal{I}$. Note that $u \wedge \mathcal{I} = \mathcal{I}$ if and only if $u \in \mathcal{P}$ \cite{Atani2}.  
\begin{theorem}
Let $\mathcal{I}, \mathcal{P}$ and $ \mathcal{Q}$  be    filters of $\mathcal{L}$ such that $\mathcal{I} \subsetneq \mathcal{P} \subsetneq \mathcal{Q}$. Then, $ \mathcal{P}$ is a   $\mathcal{Q}$-filter of $\mathcal{L}$ if and only if  $\mathcal{P}/ \mathcal{I}$ is a  $\mathcal{Q}/ \mathcal{I}$-filter of $\mathcal{L}/ \mathcal{I}$.
\end{theorem}
\begin{proof}
$\Longrightarrow$ Let $ \mathcal{P}$ be a   $\mathcal{Q}$-filter of $\mathcal{L}$. Suppose that $(u \vee v) \wedge \mathcal{I}=(u \wedge \mathcal{I}) \vee_Q (v \wedge \mathcal{I}) \in \mathcal{P}/ \mathcal{I}$ such that $u \wedge \mathcal{I} \notin \mathcal{Q}/ \mathcal{I}$. This implies that $u \vee v \in  \mathcal{P}$ with $u \notin \mathcal{Q}$. Hence, we conclude that $v \in \mathcal{P}$ as $ \mathcal{P}$ is  a  $\mathcal{Q}$-filter of $\mathcal{L}$. This means that $v \wedge  \mathcal{I} \in \mathcal{P}/ \mathcal{I}$. Consequently, $\mathcal{P}/ \mathcal{I}$ is a  $\mathcal{Q}/ \mathcal{I}$-filter of $\mathcal{L}/ \mathcal{I}$.

$\Longleftarrow$ Suppose that $\mathcal{P}/ \mathcal{I}$ is a  $\mathcal{Q}/ \mathcal{I}$-filter of $\mathcal{L}/ \mathcal{I}$. Let $u \vee v \in \mathcal{P}$ such that $u \notin \mathcal{Q}$. Therefore we have, $(u \wedge \mathcal{I}) \vee_Q (v \wedge \mathcal{I})=(u \vee v) \wedge \mathcal{I} \in \mathcal{P}/ \mathcal{I}$ and  $u \wedge \mathcal{I} \notin \mathcal{Q}/ \mathcal{I}$. This   implies that $v \wedge \mathcal{I} \in \mathcal{P}/ \mathcal{I}$ as $\mathcal{P}/ \mathcal{I}$ is a  $\mathcal{Q}/ \mathcal{I}$-filter of $\mathcal{L}/ \mathcal{I}$. Hence, we obtain $v \in \mathcal{P}$. Thus, $ \mathcal{P}$ is a   $\mathcal{Q}$-filter of $\mathcal{L}$.
\end{proof}

\section{$S$-$\mathcal{Q}$-filters}
\begin{definition}
Let $S$ be a $\vee$-closed subset of $\mathcal{L}$ and let $\mathcal{P} \subseteq \mathcal{Q}$ be filters of $\mathcal{L}$ with $S \cap \mathcal{Q}=\varnothing$. Then, $\mathcal{P}$ is called an {\it $S$-$\mathcal{Q}$-filter} of  $\mathcal{L}$ if there exists an element $s \in S$ such that for all $u,v \in \mathcal{L}$ if $u \vee v \in \mathcal{P}$, then $s \vee u \in \mathcal{P}$ or $s \vee v \in \mathcal{Q}$. In this case, we say that $\mathcal{P}$ is associated with $s$.
\end{definition}
Although every $\mathcal{Q}$-filter of $\mathcal{L}$ is an $S$-$\mathcal{Q}$-filter, the following example shows that the converse   may not be true, in general.
\begin{example} \label{exasq}
In Example \ref{exa1}(2), $S=\{0,u\}$ is a $\vee$-closed subsete of $\mathcal{L}$ and $\mathcal{P}=\{1,w\}$ is an $S$-$\mathcal{Q}$-filter of $\mathcal{L}$ where $Q=\{1,w,v\}$. However, Example \ref{exa1} verifies that $\mathcal{P}$ is not a $\mathcal{Q}$-filter of $\mathcal{L}$.
\end{example} 
\begin{remark} \label{21}
Assume  that $S$ is a $\vee$-closed subset of $\mathcal{L}$ and  $\mathcal{P} \subseteq \mathcal{Q} \subseteq \mathcal{T}$ are filters of $\mathcal{L}$ such that  $S \cap \mathcal{T}=\varnothing$. If $\mathcal{P}$ is an $S$-$\mathcal{Q}$-filter of $\mathcal{L}$, then $\mathcal{P}$ is an $S$-$\mathcal{T}$-filter.
\end{remark}
\begin{theorem} \label{22}
Let $S$ be a $\vee$-closed subset of $\mathcal{L}$ and  $\mathcal{P} \subseteq \mathcal{Q}$ be filters of $\mathcal{L}$ such that  $S \cap \mathcal{Q}=\varnothing$. Then,  $\mathcal{P}$ is an  $S$-$\mathcal{Q}$-filter of $\mathcal{L}$ if and only if there exists an $s \in S$ such that 
 for any filters $\mathcal{I}$ and $\mathcal{J}$  of $\mathcal{L}$, $\mathcal{I} \vee \mathcal{J} \subseteq \mathcal{P}$ implies  $s \vee \mathcal{I}  \subseteq \mathcal{P} $ or $s \vee \mathcal{J} \subseteq \mathcal{Q}$.

\end{theorem}
\begin{proof}
  $\Longrightarrow$   Assume that  $\mathcal{P}$ is an   $S$-$\mathcal{Q}$-filter of $\mathcal{L}$ associated with $s$. Let  $\mathcal{I}$ and $\mathcal{J}$  be filters of $\mathcal{L}$ satisfying $\mathcal{I} \vee \mathcal{J} \subseteq \mathcal{P}$ but $s \vee \mathcal{I}  \nsubseteq \mathcal{P} $.  Then, there exists $u \in \mathcal{I}$ such that $ s \vee u \notin  \mathcal{P}$. Now, take any $v \in \mathcal{J}$.  Since   $ u \vee v \in \mathcal{I} \vee \mathcal{J} \subseteq \mathcal{P}$ and $s \vee u \notin \mathcal{P}$, we have  $s \vee v \in \mathcal{Q}$ which means  $s \vee \mathcal{J}  \subseteq \mathcal{Q} $.

(ii) $\Longrightarrow$ (i) Assume that $u \vee v \in \mathcal{P}$ for $u,v \in \mathcal{L}$. Let us put $\mathcal{I}=T(\{u\})$ and $\mathcal{J}=T(\{v\})$.   From  $\mathcal{I} \vee \mathcal{J} \subseteq \mathcal{P}$ it follows that  $s \vee \mathcal{I} \subseteq \mathcal{P}$ or  $s \vee \mathcal{J} \subseteq \mathcal{Q}$ by the hypothesis. Therefore, we obtain $s \vee u \in \mathcal{P}$ or  $s \vee v \in \mathcal{Q}$. Thus, $\mathcal{P}$ is an  $S$-$\mathcal{Q}$-filter of $\mathcal{L}$.
\end{proof}
Recall from \cite{Atani2} that a filter $\mathcal{P}$ of $\mathcal{L}$ disjoin with the $\vee$-closed subset $S$ is called an $S$-prime filter if there exists an $s \in S$ such that for all $u,v \in \mathcal{L}$ if $u \vee v \in \mathcal{P}$, then $s \vee u \in \mathcal{P}$ or $s \vee v \in \mathcal{P}$.
\begin{theorem}
Let  $\mathcal{P}$ and $ \mathcal{Q}$ be filters of $\mathcal{L}$ and $S$ be a $\vee$-closed subset of $\mathcal{L}$. If $\mathcal{P} \vee \mathcal{Q}$ is an $S$-$\mathcal{Q}$-filter of $\mathcal{L}$, then there exists an $s \in S$ such that $s \vee \mathcal{P} \subseteq \mathcal{Q}$ or $\mathcal{Q}$ is an $S$-prime filter of $\mathcal{L}$ and $s \vee \mathcal{Q} \subseteq \mathcal{P} \vee \mathcal{Q}$.
\end{theorem}
\begin{proof}
Let $\mathcal{P} \vee \mathcal{Q}$ be an $S$-$\mathcal{Q}$-filter of $\mathcal{L}$ associated with $s$. Let $s \vee \mathcal{P} \nsubseteq \mathcal{Q}$. Then, we conclude that  $s \vee u \notin  \mathcal{Q}$ for some $u \in  \mathcal{P}$. Take any $v \in \mathcal{Q}$. Hence, we get $s \vee v \in \mathcal{P} \vee \mathcal{Q}$ as $\mathcal{P} \vee \mathcal{Q}$ is an $S$-$\mathcal{Q}$-filter of $\mathcal{L}$ associated with $s$, $u \vee v \in \mathcal{P} \vee \mathcal{Q}$ and $s \vee u \notin  \mathcal{Q}$.  This implies that $s \vee \mathcal{Q} \subseteq \mathcal{P} \vee \mathcal{Q}$. Now, assume that $x \vee y \in \mathcal{Q}$ for $x,y \in \mathcal{L}$. Therefore, we have $(s \vee x) \vee (s \vee y)= (s \vee s ) \vee (x \vee y) \in s \vee \mathcal{Q} \subseteq \mathcal{P} \vee \mathcal{Q}$ which mean $s \vee x=s \vee (s \vee x) \in \mathcal{P} \vee \mathcal{Q} \subseteq \mathcal{Q}$ or $s \vee y =s \vee (s \vee y) \in \mathcal{P} \vee \mathcal{Q} \subseteq \mathcal{Q}$ by the hypothesis. Thus, $\mathcal{Q}$ is an $S$-prime filter of $\mathcal{L}$.
\end{proof}
\begin{theorem}
Let $S \subseteq S^{\prime}$ be   $\vee$-closed subsets of $\mathcal{L}$ and $\mathcal{P} \subseteq \mathcal{Q}$ be filters of $\mathcal{L}$. If  $\mathcal{P} $ is an  $S^{\prime}$-$\mathcal{Q}$-filter  of $\mathcal{L}$ and  for every $r \in S^{\prime}$, there exists an $r^{\prime} \in S^{\prime}$ with $r \vee r^{\prime} \in S$, then $\mathcal{P} $ is an $S$-$\mathcal{Q}$-filter of $\mathcal{L}$.
\end{theorem}
\begin{proof}
Assume that $\mathcal{P} $ is an  $S^{\prime}$-$\mathcal{Q}$-filter  of $\mathcal{L}$ associated with $r$. Let $u \vee v \in \mathcal{P} $ for $u,v \in  \mathcal{L}$. Then, we conclude that $r \vee u \in \mathcal{P} $ or $r \vee v \in \mathcal{Q} $. By the hypothesis, we have $s=r \vee r^{\prime} \in S$ for some $r^{\prime} \in S^{\prime}$. Therefore, we obtain $s \vee u \in \mathcal{P} $ or $s \vee v \in \mathcal{Q} $ which means $\mathcal{P} $ is an $S$-$\mathcal{Q}$-filter of $\mathcal{L}$.
\end{proof}
Assume that  $(\mathcal{L}_1,\le_1)$ and $(\mathcal{L}_2,\le_2)$ are two lattices. Consider the partial order $\le_c$ on $\mathcal{L}_1 \times  \mathcal{L}_2$ defined as follows: $(u_1,u_2) \le_c (v_1,v_2)$ for all $(u_1, u_2), (v_1, v_2) \in \mathcal{L}_1  \times \mathcal{L}_2$ if and only if $u_1 \le_1 v_1$ and $u_2 \le_2 v_2$. In this case, $(\mathcal{L}_1 \times  \mathcal{L}_2,\le_c)$ is a lattice such that $(u_1,u_2) \vee_c (v_1,v_2)=(u_1 \vee v_1,u_2 \vee v_2)$ and $(u_1,u_2) \wedge_c (v_1, v_2)=(u_1 \wedge v_1, u_2 \wedge v_2)$ for every $(u_1, u_2), (v_1, v_2) \in \mathcal{L}_1 \times   \mathcal{L}_2$. The lattice  $\mathcal{L}=\mathcal{L}_1  \times \mathcal{L}_2$ is called a {\it decomposable lattice} \cite{Atani2}.

\begin{theorem} \label{car}
Let $S_1$ and $S_2$ be $\vee$-closed subsets of  the lattices $\mathcal{L}_1$ and $\mathcal{L}_2$, respectively. Let $\mathcal{P}_1 \subseteq \mathcal{Q}_1$ and  $\mathcal{P}_2 \subseteq \mathcal{Q}_2$ be proper filters of $\mathcal{L}_1$ and $\mathcal{L}_2$, respectively. If   $\mathcal{P}_1$ is an $S_1$-$\mathcal{Q}_1$-filter of $\mathcal{L}_1$ and $\mathcal{Q}_2 \cap S_2  \neq \varnothing$ or $\mathcal{P}_2$ is an $S_2$-$\mathcal{Q}_2$-filter of $\mathcal{L}_2$ and $\mathcal{Q}_1 \cap S_1 \neq \varnothing$, then $\mathcal{P}_1 \times \mathcal{P}_2$ is an $S$-$\mathcal{Q}$-filter of $\mathcal{L}_1 \times \mathcal{L}_2$ where $S= S_1 \times S_2$, $\mathcal{Q}=\mathcal{Q}_1 \times \mathcal{Q}_2$.
\end{theorem}
\begin{proof}
Let $\mathcal{P}_1$ be an $S_1$-$\mathcal{Q}_1$-filter of $\mathcal{L}_1$ associated with $s_1$ and $\mathcal{Q}_2 \cap S_2  \neq \varnothing$. This implies that   $  (\mathcal{Q}_1 \times \mathcal{Q}_2) \cap (S_1 \times  S_2) =\varnothing$.  Assume that $(u_1 \vee v_1,u_2 \vee v_2)=(u_1,u_2) \vee_c (v_1 , v_2) \in \mathcal{P}_1 \times \mathcal{P}_2$ for $(u_1,u_2),  (v_1 , v_2) \in \mathcal{L}_1 \times \mathcal{L}_2$. From $u_1 \vee v_1 \in \mathcal{P}_1$ it follows that  $s_1 \vee u_1 \in \mathcal{P}_1$ or $s_1 \vee v_1 \in \mathcal{Q}_1$. Since $\mathcal{Q}_2 \cap S_2 \neq \varnothing$, there exists $s_2 \in \mathcal{Q}_2 \cap S_2$. Therefore, we obtain $(s_1,s_2) \vee_c (u_1,u_2)=(s_1 \vee u_1,s_2 \vee u_2) \in \mathcal{Q}_1 \times \mathcal{Q}_2$ or $(s_1,s_2) \vee_c (v_1,v_2)=(s_1 \vee v_1,s_2 \vee v_2) \in \mathcal{Q}_1 \times \mathcal{Q}_2$. Consequently, $\mathcal{P}_1 \times \mathcal{P}_2$ is an $S$-$\mathcal{Q}$-filter of $\mathcal{L}_1 \times \mathcal{L}_2$. By a similar argument, we conclude that if $\mathcal{P}_2$ is an $S_2$-$\mathcal{Q}_2$-filter of $\mathcal{L}_2$ and $\mathcal{Q}_1 \cap S_1 \neq \varnothing$, then $\mathcal{P}_1 \times \mathcal{P}_2$ is an $S$-$\mathcal{Q}$-filter of $\mathcal{L}_1 \times \mathcal{L}_2$.

\end{proof}
\section{Almost $S$-$\mathcal{Q}$-filters}
\begin{definition}
Let $S$ be a $\vee$-closed subset of $\mathcal{L}$ and let $\mathcal{P} \subseteq \mathcal{Q}$ be filters of $\mathcal{L}$ with $S \cap \mathcal{P}=\varnothing$. We say that $\mathcal{P}$ is   an {\it almost $S$-$\mathcal{Q}$-filter} of  $\mathcal{L}$ if  for all $u,v \in \mathcal{L}$,  $u \vee v \in \mathcal{P}$ implies that there exists an element $s \in S$ such that $s \vee u \in \mathcal{P}$ or $s \vee v \in \mathcal{Q}$. 
\end{definition}
Assume that $S$ is a $\vee$-closed subset of $\mathcal{L}$. A filter $\mathcal{P}$ of $\mathcal{L}$  is called $S$-finite if $s \vee \mathcal{P} \subseteq \mathcal{G} \subseteq \mathcal{P}$ for some finitely generated filter $\mathcal{G}$ of $\mathcal{L}$ and $s \in S$.
\begin{theorem} \label{31}
Let $S$ be a $\vee$-closed subset of $\mathcal{L}$ and  $\mathcal{P} \subseteq \mathcal{Q}$ be filters of $\mathcal{L}$ such that  $S \cap \mathcal{P}=\varnothing$. Then,  $\mathcal{P}$ is an  almost $S$-$\mathcal{Q}$-filter of $\mathcal{L}$ if and only if  
 for all $S$-finite filters $\mathcal{I}$ and $\mathcal{J}$  of $\mathcal{L}$, $\mathcal{I} \vee \mathcal{J} \subseteq \mathcal{P}$ implies that there exists an $s \in S$ such that either  $s \vee \mathcal{I}  \subseteq \mathcal{P} $ or $s \vee \mathcal{J} \subseteq \mathcal{Q}$.
\end{theorem}
\begin{proof}
$\Longrightarrow$ Let $\mathcal{P}$ be an  almost $S$-$\mathcal{Q}$-filter of $\mathcal{L}$. Assume that $\mathcal{I}$ and $\mathcal{J}$ are $S$-finite filters of $\mathcal{L}$ satisfying $\mathcal{I} \vee \mathcal{J} \subseteq \mathcal{P}$ but $s \vee \mathcal{J} \nsubseteq \mathcal{Q}$ for all $s \in S$. Then, we conclude that $s_1 \vee \mathcal{I} \subseteq T(\{u_1,\ldots,u_n\})$ and $s_2 \vee \mathcal{J} \subseteq T(\{v_1,\ldots,v_m\})$ for some $s_1,s_2 \in S$, $u_1,\ldots,u_n \in \mathcal{I}$ and $v_1,\ldots,v_m \in \mathcal{J}$. If $s \vee T(v_1,\ldots,v_m\}) \subseteq \mathcal{Q}$ for some $s \in S$, then we obtain $(s \vee  s_2) \vee \mathcal{J} \subseteq s \vee T(\{v_1,\ldots,v_m\}) \subseteq \mathcal{Q}$, a contradiction.
 Hence, we get $s \vee T(v_1,\ldots,v_m\}) \nsubseteq \mathcal{Q}$ for all $s \in S$. This implies that $s \vee v_j \notin \mathcal{Q}$ for some $1 \leq j \leq m$. Assume that $1 \leq i \leq n$. Therefor, we have $u_i \vee v_j \in \mathcal{P}$ but $s \vee v_j \notin \mathcal{Q}$ for all $s \in S$.   Since $\mathcal{P}$ is  an  almost $S$-$\mathcal{Q}$-filter of $\mathcal{L}$, there exists $t_i \in S$ with $t_i \vee u_i  \in  \mathcal{P}$ for each $1 \leq i \leq n$. Put $t=t_1 \vee \ldots \vee t_n$. So, we conclude that $(t \vee s_1)\vee \mathcal{I} \subseteq t \vee T(\{u_1,\ldots,u_n\}) \subseteq \mathcal{P}$. Set $s=t \vee s_1$. Consequently, $s \vee \mathcal{I} \subseteq \mathcal{P}$.
 
 $\Longleftarrow$ Suppose that  $u \vee v \in \mathcal{P}$ for $u,v \in \mathcal{L}$. Set $\mathcal{I}=T(\{u\})$ and $\mathcal{J}=T(\{v\})$.  Since  $\mathcal{I} \vee \mathcal{J} \subseteq \mathcal{P}$ it follows that there exists an $s \in S$ such that  $s \vee \mathcal{I} \subseteq \mathcal{P}$ or  $s \vee \mathcal{J} \subseteq \mathcal{Q}$. This means that,  $s \vee u \in \mathcal{P}$ or  $s \vee v \in \mathcal{Q}$. Thus, $\mathcal{P}$ is an  almost $S$-$\mathcal{Q}$-filter of $\mathcal{L}$.
\end{proof}
\begin{theorem}
Let $S$ be a $\vee$-closed subset of $\mathcal{L}$ and let $\mathcal{I} \subseteq \mathcal{P} \subseteq \mathcal{Q}$  be    filters of $\mathcal{L}$ such that $S \cap \mathcal{P}=\varnothing$. Then, $ \mathcal{P}$ is an almost   $S$-$\mathcal{Q}$-filter of $\mathcal{L}$ if and only if  $\mathcal{P}/ \mathcal{I}$ is an almost $\bar{S}$-$\mathcal{Q}/ \mathcal{I}$-filter of $\mathcal{L}/ \mathcal{I}$ such that $\bar{S}=\{s \wedge \mathcal{I} \mid s \in S\}$.
\end{theorem}
\begin{proof}
$\Longrightarrow$ Let $ \mathcal{P}$ be an almost $S$-$\mathcal{Q}$-filter of $\mathcal{L}$. Assume  that $(u \vee v) \wedge \mathcal{I}=(u \wedge \mathcal{I})\vee_Q (v \wedge \mathcal{I})  \in \mathcal{P}/ \mathcal{I}$ for $u \wedge \mathcal{I} , v \wedge \mathcal{I} \in \mathcal{L}/ \mathcal{I}$ .  This means that  $u \vee v \in  \mathcal{P}$. Therefore, there exists some $s \in S$ such that   $s \vee v \in \mathcal{P}$ or $s \vee v \in \mathcal{Q}$ as $ \mathcal{P}$ is  an almost $S$-$\mathcal{Q}$-filter of $\mathcal{L}$. It yields that  $(s \vee u)\wedge \mathcal{I}=(s \wedge \mathcal{I})\vee_Q (u \wedge \mathcal{I}) \in \mathcal{P}/ \mathcal{I}$ or $(s \vee v)\wedge \mathcal{I}=(s \wedge \mathcal{I})\vee_Q (v \wedge \mathcal{I}) \in \mathcal{Q}/ \mathcal{I}$. Thus, $\mathcal{P}/ \mathcal{I}$ is an almost $\bar{S}$-$\mathcal{Q}/ \mathcal{I}$-filter of $\mathcal{L}/ \mathcal{I}$ 

$\Longleftarrow$ Assume that $\mathcal{P}/ \mathcal{I}$ is an almost  $\mathcal{Q}/ \mathcal{I}$-filter of $\mathcal{L}/ \mathcal{I}$ and $u \vee v \in \mathcal{P}$. This means that  $(u \vee v)  \wedge \mathcal{I} =(u  \wedge \mathcal{I})  \vee_Q (v  \wedge \mathcal{I}) \in \mathcal{P}/ \mathcal{I}$. Therefore,    there exists $s  \wedge \mathcal{I} \in \bar{S}$ such that $(s \wedge \mathcal{I}) \vee_Q (u\wedge \mathcal{I})=(s \vee u)\wedge \mathcal{I} \in \mathcal{P}/ \mathcal{I}$ or $(s \wedge \mathcal{I}) \vee_Q (v\wedge \mathcal{I})=(s \vee v)\wedge \mathcal{I} \in \mathcal{Q}/ \mathcal{I}$ as $\mathcal{P}/ \mathcal{I}$ is an almost $\bar{S}$-$\mathcal{Q}/ \mathcal{I}$-filter of $\mathcal{L}/ \mathcal{I}$. Hence,  we have  $s \vee u \in \mathcal{P}$ or $s \vee v \in \mathcal{Q}$. Consequently, $ \mathcal{P}$ is an  almost $S$-$\mathcal{Q}$-filter of $\mathcal{L}$.
\end{proof}


\end{document}